\include{kbordermatrix.sty}
\documentclass[12pt]{amsart}
\usepackage{amsmath,amssymb,enumerate}
\newtheorem{theorem}{Theorem}[section]

\newtheorem{lemma}[theorem]{Lemma}

\newtheorem{conjecture}[theorem]{Conjecture}
\theoremstyle{definition}

\newcommand{\cC}{\mathcal{C}}

\newcommand{\cJ}{\mathcal{J}}
\newcommand{\cL}{\mathcal{L}}

\newcommand{\cP}{\mathcal{P}}

\newcommand{\cU}{\mathcal{U}}

\DeclareMathOperator{\PG}{PG}
\DeclareMathOperator{\GF}{GF}

\newcommand{\del}{\setminus \!}
\newcommand{\con}{/}
\newcommand{\gbinom}[3]{{{#1} \brack {#2}}_{#3}}
\newcommand{\qbinom}[2]{\gbinom{#1}{#2}{q}}
\newcommand{\tqbinom}[2]{\genfrac{[}{]}{0pt}{1}{#1}{#2}_q}
\allowdisplaybreaks[1]

\author{Jim Geelen}
\author{Peter Nelson}
\address{Department of Combinatorics and Optimization,
University of Waterloo, Waterloo, Canada}
\thanks{ This research was partially supported by a grant from the
Office of Naval Research [N00014-12-1-0031].}
\dedicatory{This paper is dedicated to the memory of Michel Las Vergnas.}
\title[The number of lines in a matroid]{The Number of Lines in a Matroid with no $U_{2,n}$-minor}
\begin{document}
\begin{abstract}
We show that, if $q$ is a prime power at most $5$, then every rank-$r$ matroid with no $U_{2,q+2}$-minor has no more lines than a rank-$r$ projective geometry over $\GF(q)$. We also give examples showing that for every other prime power this bound does not hold. 
\end{abstract}
\maketitle

\section{Introduction}

This paper is motivated by the following special case of a conjecture
due to Bonin; see Oxley~[\ref{oxley}, p. 582].
\begin{conjecture}\label{bonin}
For each prime power $q$ and positive integer $r$,every rank-$r$ matroid with no $U_{2,q+2}$-minor
has at most $\qbinom{r}{2}$ lines.
\end{conjecture}
Here $\qbinom{r}{2} = \tfrac{(q^r-1)(q^{r-1}-1)}{(q-1)(q^2-1)}$ is a $q$-binomial coefficient. The projective geometry $\PG(r-1,q)$ has $\qbinom{r}{2}$ lines,
so the conjectured bound is attained. Blokhuis gave examples refuting Conjecture~\ref{bonin}
for all $q\ge 13$; see Nelson~[\ref{nelson}]. Our main result is the following.
\begin{theorem}\label{main}
Conjecture~\ref{bonin} holds if and only if $q\le 5$.
\end{theorem}
All known counterexamples to Conjecture~\ref{bonin} have rank $3$ and it is quite plausible that the 
conjecture holds whenever $r\ge 4$; this is supported by a result of Nelson~[\ref{nelson}] that the conjecture
holds when $r$ is sufficiently large relative to $q$.

The proof of Conjecture~\ref{bonin} is straightforward for $q \in \{2,3,4\}$. For $q = 5$ we solve the problem partly by computer search. In all four cases we devote most of our attention to the rank $3$ case, to which the general case is easily reduced. 

\section{Preliminaries}

We follow the notation of Oxley [\ref{oxley}]. We write $\cU(\ell)$ for the class of matroids with no $U_{2,\ell+2}$-minor. If $e \in E(M)$ then we write $W_1(M)$ for the number of points of $M$, $W_2(M)$ for the number of lines of $M$, $W_2^e(M)$ for the number of lines of $M$ \emph{not} containing $e$, and $\delta_M(e)$ for the number of lines of $M$ containing $e$. For a simple rank-$3$ matroid $M$, we have $M \in \cU(\ell)$ iff $\delta_M(e) \le \ell + 2$ for all $e \in E(M)$. $W_1$ and $W_2$ are the first two  \emph{Whitney numbers of the second kind}.

The following theorem was proved by Kung [\ref{kungextremal}].

\begin{theorem}\label{kung}
	If $\ell \ge 2$ is an integer and $M \in \cU(\ell)$ has rank $r$, then $W_1(M) \le \tqbinom{r}{1} = \frac{q^{r}-1}{q-1}$.
\end{theorem}

Surprisingly, we require a small graph theory result. A \emph{$1$-factorisation} of a graph is a partition of its edge set into perfect matchings.

\begin{lemma}\label{k6}
	Any two $1$-factorisations of the graph $K_6$ have an element in common. 
\end{lemma}
\begin{proof}
	A $1$-factorisation of $K_6$ is a $5$-edge-colouring. The union of any two colour classes is a $2$-regular bipartite graph on $6$ vertices and edges, so is a $6$-cycle, and it is easy to check that for any $6$-cycle $C$ there is a unique $5$-edge-colouring having $C$ as the union of two of its colour classes. Each $5$-edge-colouring has $10$ pairs of colour classes and $K_6$ has $60$ $6$-cycles, so $K_6$ has six $1$-factorisations. 
	
	Suppose that there exist disjoint $1$-factorisations $F_1$ and $F_2$. Each edge is in exactly three perfect matchings, so the set $F_3$ of perfect matchings not in $F_1$ or $F_2$ is also a $1$-factorisation. Let $F$ be a $1$-factorisation that is not $F_1$, $F_2$ or $F_3$. Since $|F| = 5$ there is some $i$ such that $|F \cap F_i| \ge 2$, but now $F$ and $F_i$ share two colour classes and are thus equal by our above observation. This is a contradiction. 	
\end{proof}

Our next lemma, invoked twice in Section~\ref{five}, was proved by a computer search whose structure we briefly sketch. 

\begin{lemma}\label{computation}
	Let $A$ be a twelve-element set. There do not exist partitions $\cL_0,\cL_1, \dotsc, \cL_5$ of $A$ satisfying the following conditions: 
	\begin{enumerate}
		\item $\cL_0$ has exactly six blocks, each of size $2$,
		\item for each $i \in \{1,\dotsc,5\}$, the partition $\cL_i$ has at most $5$ blocks and each has size at most $4$, 
		\item for every distinct $x,y \in A$, there is exactly one $i \in \{0, \dotsc, 6\}$ such that $\cL_i$ has a block containing $x$ and $y$, 
		\item for each $i \in \{1, \dotsc, 5\}$, if $\cL_i$ has exactly five blocks then it has a block of size $1$. 
\end{enumerate} 
\begin{proof}[Sketch of computational proof:]
Fix $\cL_0$ arbitrarily and suppose that partitions $\cL_1, \dotsc, \cL_5$ exist. For convenience we assume they each have exactly five parts and allow parts to be empty. The block sizes of each $\cL_i : i \in \{1, \dotsc, 5\}$ gives an integer partition $(n_{i,1},\dotsc,n_{i,5})$ of $12$ so that $4 \ge n_{i,1} \ge n_{2} \ge \dotsc \ge n_{i,5} \ge 0$ and $n_{i,5} \le 1$. Moreover, there are $66$ unordered pairs of distinct elements of $A$ and six of these pairs are contained in blocks of $\cL_0$, so $\sum_{i = 1}^5 \sum_{j = 1}^5 \binom{n_{i,j}}{2} = 60$. 

	We say two set partitions $P,P'$ are \emph{compatible} if each block of $P$ intersects each block of $P'$ in at most one element. For each integer partition $p$ of $12$ into nonnegative parts, let $C(p)$ denote the set of partitions of $A$ that are compatible with $\cL_0$ and whose block sizes are the integers in $p$. Let $C'(p)$ denote the set of orbits of $C(p)$ under the action of the group of the $6! \cdot 2^6$ permutations of $A$ that fix $\cL_0$. The following table shows the nine possible $p$ that satisfy our constraints and their associated parameters.
	\begin{center}
	 \begin{tabular}{| c | c | c | c |} 
	 	\hline
		$p$ & $|C(p)|$ & $|C'(p)|$ & $\sum_{j=1}^5 \tbinom{p_j}{2}$\\
		\hline 
		$(3, 3, 3, 2, 1)$ & 71040 & 5 & 10\\
		$(3, 3, 3, 3, 0)$ &  4960 & 3 & 12\\
		$(4, 3, 2, 2, 1)$ & 136320 & 9 & 11\\
		$(4, 3, 3, 1, 1)$ & 41280 & 5 & 12\\
		$(4, 3, 3, 2, 0)$ & 38400 & 4 & 13 \\
		$(4, 4, 2, 1, 1)$ & 27360 & 5 & 13\\
		$(4, 4, 2, 2, 0)$ & 12720 & 4 & 14\\
		$(4, 4, 3, 1, 0)$ & 15360 & 2 & 15\\
		$(4, 4, 4, 0, 0)$ & 960 & 1 & 18\\
		\hline
	 \end{tabular}
	\end{center}
	The tuple $(\cL_1, \dotsc, \cL_5)$ must belong to $\cC = C(p_1) \times C(p_1) \times \dotsc \times C(p_5)$, where $p_1, \dotsc, p_5$ are drawn from rows of the table above whose last column sums to $60$; there are $68$ such (unordered) $5$-tuples $p_1, \dotsc, p_5$. Moreover, the partitions $\cL_0, \dotsc, \cL_5$ must be pairwise compatible. For each of the $68$ possible $\cC$, a backtracking search shows this cannot occur; by considering our choice for $\cL_1$ up to a permutation of $A$ that preserves $\cL_0$, we need only consider one choice of $\cL_1$ from each orbit in $C'(p_1)$. Our search was performed with a Python program that runs in under two hours on a single CPU.

\end{proof}
	
	%$p_{i,j}: i \in \{1, \dotsc 5\}, j \in \{1, \dotsc, 5\}$ 
\end{lemma}

\section{Counterexamples}

In this section we construct counterexamples to Conjecture~\ref{bonin}. They are more elaborate versions of the aforementioned construction of Blokhuis. 

\begin{lemma}\label{construction}
	Let $q$ be a prime power and $t$ be an integer with $3 \le t \le q$. There is a rank-$3$ matroid $M(q,t)$ with no $U_{2,q+t}$-minor such that $W_2(M(q,t)) = q^2 + (q+1)\tbinom{t}{2}$. 
\end{lemma}
\begin{proof}
	Let $N \cong \PG(2,q)$. Let $e \in E(N)$ and let $L_1,L_2,L_3$ be distinct lines of $N$ not containing $e$ and so that $L_1 \cap L_2 \cap L_3$ is empty. Note that every line of $M$ other than $L_1, L_2$ and $L_3$ intersects $L_1 \cup L_2 \cup L_3$ in at least $2$ and at most $3$ elements. 
	
	Let $\cL$ be the set of lines of $N$ and $\cL_e$ be the set of lines of $N$ containing $e$. For each $L \in \cL_e$, let $T(L)$ be a $t$-element subset of $L - \{e\}$ containing $L \cap (L_1 \cup L_2 \cup L_3)$. Observe that the $T(L)$ are pairwise disjoint. Let $X = \cup_{L \in \cL_e} T(L)$, noting that $L_1 \cup L_2 \cup L_3 \subseteq X$ and so each line in $\cL$ intersects $X$ in at least two elements. Let $M(q,t)$ be the simple rank-$3$ matroid with ground set $X$ whose set of lines is $\cL_1 \cup \cL_2$, where $\cL_1 = \{L \cap X: L \in \cL - \cL_e\}$, and $\cL_2$ is the collection of two-element subsets of the sets $T(L): L \in \cL_e$. Note that $\cL_1$ and $\cL_2$ are disjoint. Every $f \in X$ lies in $q$ lines in $\cL_1$ and in $(t-1)$ lines in $\cL_2$, so $M(q,t)$ has no $U_{2,q+t}$-minor. Moreover, we have $\cL_1 = |\cL - \cL_e| = q^2$ and $\cL_2 = |\cL_e|\binom{t}{2} = (q+1)\binom{t}{2}$. This gives the lemma. 
	\end{proof}

This next theorem refutes Conjecture~\ref{bonin} for all $q \ge 7$. 

\begin{theorem}
	If $\ell$ is an integer with $\ell \ge 7$, then there exists $M \in \cU(\ell)$ such that $r(M) = 3$ and $W_2(M) > \ell^2 + \ell + 1$. 
\end{theorem}
\begin{proof}
	 If $\ell \ge 127$, let $q$ be a power of $2$ such that $\tfrac{1}{4}(\ell+2) < q \le \tfrac{1}{2}(\ell+2)$.  We have $W_2(M(q,q)) = q^2 + \binom{q}{2}(q+1) >\tfrac{1}{2}q^3 \ge \tfrac{1}{128} (\ell+1)^3 \ge (\ell+1)^2 > \ell^2 + \ell + 1$. 
	
	%If $\ell = 8$, then $q = 5$ and $W_2(M(q,q)) = 85 > 8^2 + 8 + 1$. If $9 \le \ell \le 11$, then $q \ge 7$ and $W_2(M(q,q)) \ge 217 > \ell^2 + \ell + 1$. If $12 \le \ell \le 20$ then 
	If $7 \le \ell < 127$, then it is easy to check that there is some prime power $q \in \{5,7,9,13,19,32,59,113\}$ such that $\tfrac{1}{2}(\ell+2) \le q \le \ell-2$. Note that $3 < \ell+2-q \le q$. Let $f_q(x) = q^2 + (q+1)\tbinom{x+2-q}{2} - (x^2 + x + 1)$. This function $f_q(x)$ is quadratic in $x$ with positive leading coefficient and $f_q(q) = f_q(q+1) = 0$; it follows that $f(x) > 0$ for every integer $x \notin \{q,q+1\}$. Now the matroid $M = M(q,\ell+2-q)$ satisfies $M \in \cU(\ell)$ and $W_2(M) - (\ell^2 + \ell + 1) = f_q(\ell) > 0$. 
\end{proof}

We conjecture that, for large $\ell$, the matroids $M(q,q)$ give the correct upper bound for the number of lines in a rank-$3$ matroid in $\cU(\ell)$.

\begin{conjecture}
	If $\ell$ is a sufficiently large integer and $M \in \cU(\ell)$ has rank $3$, then $W_2(M) \le W_2(M(q,q)) = q^2 + \binom{q}{2}(q+1)$, where $q$ is the largest prime power such that $2q \le \ell+2$.   
\end{conjecture}

\section{Small $q$}

\begin{lemma}\label{verylongline}
	Let $q \ge 2$ be an integer. If $M \in \cU(q)$ has rank $3$ and has a $U_{2,q+1}$-restriction, then $W_2(M) \le q^2 + q + 1$ and $W_2^e(M) \le q^2$ for each nonloop $e$ of $M$. 
\end{lemma}
\begin{proof}
	We may assume that $M$ is simple; let $M|L$ be a $U_{2,q+1}$-restriction of $M$. If some line $L'$ of $M$ does not intersect $L$ then contracting a point of $L'$ yields a $U_{2,q+2}$-minor, so every line of $M$ intersects $L$. Therefore $W_2(M) = \sum_{x \in L}(\delta_M(x) - 1) + 1 \le (q+1)((q+1)-1) + 1 = q^2 + q + 1$. For each $e \in E(M) - L$ we clearly have $\delta_M(e) = q+1$ so $W_2^e(M) \le (q^2 + q + 1) - (q + 1) = q^2$. For each $e \in L$ we have $W_2^e(M) = \sum_{x \in L - \{e\}} (\delta_M(e) - 1) \le q (q + 1 - 1 ) = q^2$. 
\end{proof}

\begin{lemma}\label{disjointbound}
	If $q \in \{2,3,4\}$ and $M \in \cU(q)$ is a rank-$3$ matroid with a $U_{2,q}$-restriction $L$ and no $U_{2,q+1}$-restriction, then at most $q$ lines of $M$ are disjoint from $L$.
\end{lemma}
\begin{proof}
	We may assume that $M$ is simple. Suppose that there is a set $\cL$ of lines disjoint from $L$ such that $|\cL| = q+1$. Since each $x \in E(M) - L$ lies on $q$ lines intersecting $L$ it lies on at most one line in $\cL$, sothe lines in $\cL$ are pairwise  disjoint. Let $X$ be a set formed by choosing two points from each line in  $\cL$; note that $|X| = 2(q+1)$ and $X \cap L = \varnothing$. 
	
	Since each $X$ lies on at most one line disjoint from $L$, at most $(q+1)$ pairs of elements of $X$ span lines disjoint from $L$, so at least $\binom{2(q+1)}{2} - (q+1) = 2q(q+1)$ pairs of elements of $X$ span a line intersecting $L$. Since $|L| = q$, there is some $y \in L$ such that at least $2(q+1)$ pairs of elements of $X$ span $y$. Let $\cL_y$ be the set of lines of $M|(\{y\} \cup X)$ that contain $y$. Every line in $\cL_y$ spans a line of $M$ containing $y$ and none spans $L$ itself, so $|\cL_y| \le q$. We also have $\sum_{L \in \cL_y}(|L| - 1) = |X| = 2(q+1)$ and $\sum_{L \in \cL_y}\binom{|L|-1}{2} \ge 2(q+1)$ by choice of $y$. Since $M$ has no $U_{2,q+1}$-restriction, we also have $|L| - 1 \le q-1$ for each $L \in \cL_y$. It remains to check that, for $q \in \{2,3,4\}$ there are no solutions to the system $n_1 + n_2 + \dotsc + n_q = 2(q+1)$, $\binom{n_1}{2} + \dotsc + \binom{n_q}{2} \ge 2(q+1)$ subject to $n_i \in \{0,\dotsc,q-1\}$ for each $i$. This is easy.

\end{proof}

\begin{lemma}\label{longline}
	Let $q \in \{2,3,4\}$. If $M \in \cU(q)$ has rank $3$ and has a $U_{2,q}$-restriction, then $W_2(M) \le q^2 + q + 1$ and $W_2^e(M) \le q^2$ for each nonloop $e$ of $M$.
\end{lemma}
\begin{proof}
	We may assume that $M$ is simple and, by Lemma~\ref{verylongline}, that $M$ has no $U_{2,q+1}$-restriction; let $M|L$ be a $U_{2,q}$-restriction of $M$ and let $f \in L$. If $W_2^f(M) \ge q^2 + 1$ then, since each $x \in L - \{f\}$ is on at most $q$ lines not containing $f$, there are at most $(|L|-1)q = q^2 - q$ lines that intersect $L$ but not $f$. Therefore there are at least $(q^2+1)-(q^2-q) = q+1$ lines that do not intersect $L$. This is a contradiction by Lemma~\ref{disjointbound}. So $W_2^f(M) \le q^2$ for each $e$ in a $U_{2,q}$-restriction of $M$; since $W^2(M) = W_2^f(M) + \delta_M(f) \le W_2^f(M) + q+ 1$ for every $f$ this resolves the first part of the lemma, as well as the second part if $e$ is in a $U_{2,q}$-restriction. 
	
	It remains to bound $W_2^e(M)$ if $e$ is in no $U_{2,q}$-restriction. If $\delta_M(e) \ge q+1$ then we have $W_2^e(M) = W_2(M) - \delta_e(M) \le q^2$ as required, so we may assume that $\delta_M(e) \le q$. Therefore $e$ is in at most $q$ lines containing at most $q-2$ other points each, so $|E(M) - e| \le q(q-2)$. Each $x \in E(M) - e$ is in at most $q$ lines not containing $e$ and each such line contains at least $2$ points of $E(M) - e$, so $W_2^e(M) \le \tfrac{1}{2}q|E(M)-e| = \tfrac{1}{2}q^2(q-2) \le q^2$, since $\tfrac{1}{2}(q-2) \le 1$. 
	
\end{proof}

\begin{lemma}\label{shortline}
	If $q \in \{2,3,4\}$ and $M \in \cU(q)$ has rank $3$ and has no $U_{2,q}$-restriction, then $W_2(M) \le q^2 + q + 1$ and $W_2^e(M) \le q^2$ for each nonloop $e$ of $M$.
\end{lemma}
\begin{proof}
	We may assume that $M$ is simple; let $n = |M|$. If $q = 2$ then the result is vacuous and if $q = 3$ then $M$ has no $U_{2,3}$-restriction so $M \cong U_{3,n}$ and $n \le 5$ so both conclusions are clear. It remains to resolve the $q = 4$ case. 
	
	Suppose that $W_2(M) \ge 4^2 + 4 + 2 = 22$. Every line of $M$ contains either two or three points; for each $f \in E(M)$ let $\ell_f$ be the number of $3$-point lines of $M$ containing $f$. Let $\ell$ be the total number of $3$-point lines of $M$. Each $3$-point line of $M$ contains $3$ pairs of points of $M$, so $22 \le W_2(M) = \binom{n}{2} - 2\ell$. Moreover, every $e \in E(M)$ is in at most $5$ lines so $n \le 1 + 2 \ell_f + (5-\ell_f) = 6 + \ell_f$. Summing this expression over all $f \in E(M)$ gives $n^2 \le 6n + 3\ell$. Therefore $2(6n+3\ell) + 3(\binom{n}{2}-2\ell) \ge 2n^2 + 66$, giving $0 \ge n^2 - 21n + 132 = (n - \tfrac{21}{2})^2 + \tfrac{87}{4}$, a contradiction; therefore $W_2(M) \le 4^2 + 4 + 1$. From here, it is also easy to obtain a contradiction to $W_2^e(M) > 4^2$ in a manner similar to the proof of Lemma~\ref{longline}. 
\end{proof}

\section{Five}\label{five}

We now consider the number of lines in rank-$3$ matroids in $\cU(5)$, first dealing with those that have no $U_{2,5}$-restriction. 

\begin{lemma}\label{shortlines1}
	If $M \in \cU(5)$ has rank $3$ and has no $U_{2,5}$-restriction, then $W_2(M) \le 5^2 + 5 + 1$.
\end{lemma}
\begin{proof}
	We may assume that $M$ is simple. Let $n = |M|$ and for each $i \in \{2,3,4\}$, let $\ell_i$ be the number of lines of length $i$ in $M$, noting that every line of $M$ has length $2,3$ or $4$. Suppose for a contradiction that $\ell_2 + \ell_3 + \ell_4 \ge 32$. Let $P$ be the set of pairs $(e,L)$ where $e \in L$. We have $2 \ell_2 + 3 \ell_3 + 4\ell_4 = |P| = \sum_{e \in E(M)} \delta_M(e) \le 6n$. There are $\binom{n}{2}$ pairs of elements of $M$, each of which is contained in exactly one line of $M$, and an $i$-element line contains $\binom{i}{2}$ such pairs. We therefore have $\ell_2 + 3\ell_3 + 6\ell_4 = \binom{n}{2}$. Now 
	\begin{align*}
		\ell_4 &= (\ell_2 + 3\ell_3 + 6\ell_4) + 3(\ell_2 + \ell_3 + \ell_4) - 2(2\ell_2 + 3\ell_3 + 4\ell_4) \\ 
			   &\ge \tbinom{n}{2}  + 3 \cdot 32 - 2 \cdot 6n
	\end{align*}
	nd $\ell_1 + 3\ell_3 = \binom{n}{2} - 6\ell_4 \le 72n - 18 \cdot 32 - 5\binom{n}{2} = \tfrac{-5}{2}\left(n-\tfrac{149}{10}\right)^2 - \tfrac{839}{40} < 0$, a contradiction.
\end{proof}

\begin{lemma}\label{shortlines2}
	If $M \in \cU(5)$ is a rank-$3$ matroid with no $U_{2,5}$-restriction and $e$ is a nonloop of $M$, then $W_2^e(M) \le 5^2$. 
\end{lemma}
\begin{proof}
	We may assume that $M$ is simple. If $\delta_M(e) = 6$ then $W_2^e(M) \le 31-6 = 25$ by the previous lemma, so we may assume that $\delta_M(e) \le 5$. Let $n = |M|$ and let $\ell_2^e$, $\ell_3^e$ and $\ell_4^e$ be the number of lines of length $2$, $3$ and $4$ respectively that do not contain $e$. Suppose for a contradiction that $\ell_2^e + \ell_3^e + \ell_4^e \ge 26$. Let $P$ be the set of pairs $(f,L)$, where $L$ is a line not containing $e$ and $f \in L$.  Clearly $|P| = 2\ell_2^e + 3\ell_3^e + 4\ell_4^e$, but also, since every $f \ne e$ is on at most $5$ lines not containing $e$, we have $|P| \le 5(n-1)$, so $2\ell_2^e + 3\ell_3^e + 4\ell_4^e \le 5(n-1)$. Finally, let $Q$ be the set of two-element sets $\{f_1,f_2\} \subset E(M)$ that span a line not containing $e$. As before, we have $|Q| = \ell_2^e + 3\ell_3^e + 6\ell_4^e$. On the other hand, there are at most $5$ lines of $M$ through $e$ and each contains at most $3$ other points, so there are at most $5\binom{3}{2} = 15$ two-element subsets of $E(M) - \{e\}$ that are not in $Q$. Therefore $|Q| = \binom{n-1}{2}-s$ for some $s \in \{0, \dotsc, 15\}$, and $\ell_2^e + 3\ell_3^e + 6\ell_4^e = \binom{n-1}{2}-s$. Now
	\begin{align*}
		\ell_4^e &= (\ell_2^e + 3\ell_3^e + 6\ell_4^e) + 3(\ell_2^e + \ell_3^e + \ell_4^e) - 2(2\ell_2^e + 3\ell_3^e + 4\ell_4^e) \\ 
			   &\ge \tbinom{n-1}{2} - s + 3 \cdot 26 - 2(5(n-1))\\
			   &= \tbinom{n-1}{2} - 10n + 88 - s.
	\end{align*}
	Therefore, using $s \le 15$ we have $\ell_2^e + 3\ell_3^e = |Q|-6\ell_4 \le \binom{n-1}{2}-s - 6(\binom{n-1}{2}-10n+88-s) = 60n - 528 - 5\binom{n-1}{2} + 5s \le 60n -453 - 5\binom{n-1}{2} = \tfrac{-5}{2}\left(n - \tfrac{27}{2}\right)^2 - \tfrac{19}{8} < 0$, a contradiction. 
\end{proof}

\begin{lemma}\label{bash}
	If $M \in \cU(5)$ has rank $3$ and has a $U_{2,5}$-restriction, then $W_2(M) \le 5^2 + 5 + 1$.
\end{lemma}
\begin{proof}
	Let $M$ be a counterexample for which $|M|$ is minimized. Note that $M$ is simple, that $W_2(M) \ge 32$, and that, by Lemma~\ref{verylongline}, $M$ has no $U_{2,6}$-restriction.  
	
	Let $L = \{x_1,x_2,x_3,x_4,x_5\}$. Each element of $L$ lies on at most five other lines, so there are at least $32 - 5 \cdot 5 -1 = 6$ lines $L_{0,1},L_{0,2}, \dotsc, L_{0,6}$ of $M$ that do not intersect $L$.  For each $i \in \{1, \dotsc, 6\}$ let $a_{2i-1}$ and $a_{2i}$ be distinct elements of $L_{0,i}$. Note that each $e \in E(M) - L$ lies on five lines meeting $L$ so lies on at most one other line; it follows that the set $A = \{a_1, a_2, \dotsc, a_{12}\}$ has twelve elements and that $\cL_0 = \{L_{0,0}, \dotsc, L_{0,6}\}$ is a partition of $A$ into pairs. 
	
	For each $i \in \{1, \dotsc, 5\}$ let $\cL_i'$ be the set of lines of $M$ containing $x_i$ other than $L$ and let $\cL_i = \{L' - \{x_i\}: L' \in \cL_i'\}$. We have $|\cL_i| \le 5$ and clearly $\cL_i$ is a partition of $A$. If there are six lines through $x_i$ each containing at least two other points, then $W_2(M \del x_i) = W_2(M)$, contradicting minimality of $|M|$. Therefore $|L'| \le 1$ for some $L' \in \cL_i$. Since $M$ has no $U_{2,6}$-restriction we also have $|L'| \le 4$ for each $L \in \cL_i$. Finally, since each two-element subset of $A$ either spans a line in $\cL_0$ or a line in $\cL_i'$ for a unique $i$, each such pair is contained in a block of exactly one of the partitions $\cL_0, \dotsc, \cL_5$. By Lemma~\ref{computation} this is impossible.   % These lines induce a partition of $A$ into at most $5$ blocks; for convenience we assume that this partition has exactly five blocks and allow blocks to be empty. So there exist sets $L_{i,1}, \dotsc, L_{i,5}$ with disjoint union $A$ such that each $L_{i,j}$ is either empty or is obtained by removing $x_i$ from a line of $N$ through $x_i$. Let $\cL_i = \{L_{i,1}, \dotsc, L_{i,5}\}$ for each $i \in \{1,\dotsc, 5\}$. We say that two partitions of $A$ are \emph{compatible} if any pair of blocks of the two partitions intersect in at most one element; note that the partitions $\cL_0, \dotsc, \cL_5$ are pairwise compatible.
	
	  %We may assume by Lemma~\ref{verylongline} that $|L_{i,j}| \le 4$ for each $i$ and $j$. If for some $i$ we have $|L_{i,j}| \ge 2$ for all $j \in \{1, \dotsc, 5\}$ then $x_i$ lies on six lines of $M$ each having at least three points, so $W_2(M \del x_i) = W_2(M)$, contradicting minimality. We may therefore assume that $4 \ge |L_{i,1}| \ge \dotsc \ge |L_{i5}|$ and $|L_{i5}| \le 1$.  Note that $\sum_{j = 1}^5|L_{i,j}| = 12$ for each $i \in \{1, \dotsc, 5\}$. There are $\binom{12}{2} = 66$ pairs of elements of $A$ and each pair other than the blocks of $\cL_0$ is contained in exactly one $L_{i,j}$ for some $i,j \in \{1, \dotsc, 5\}$. Exactly $6$ such pairs are blocks of $\cL_0$, so there are $60$ pairs of elements that are contained in a block of one of $\cL_1, \dotsc, \cL_5$. Therefore $\sum_{i = 1}^5 \sum_{j = 1}^5 \binom{|L_{i,j}|}{2} = 60$.
		
	\end{proof}
	
\begin{lemma}
	If $M \in \cU(5)$ has rank $3$ then $W_2^e(M) \le 5^2$ for each nonloop $e$ of $M$. 
\end{lemma}
\begin{proof}
	Let $(M,e)$ be a counterexample for which $|M|$ is minimized. Note that $M$ simple and that, by Lemma~\ref{shortlines2}, $M$ has a $U_{2,5}$-restriction $M|L$. If $\delta_M(e) \ge 6$ then $W_2^e(M) \le 5^2 + 5 + 1 - 6 = 25$ by Lemma~\ref{bash}, so $\delta_M(e) \le 5$. If there is some $f \in E(M) - \{e\}$ on six lines each containing at least two other points, then $W_2^e(M \del f) = W_2^e(M)$, contradicting minimality. Therefore every $x \in E(M)$ is on at most five lines that contain two other points (note that $e$ also has this property).  

If $e \in L$ then observe that each $f \in L - \{e\}$ is on at most $5$ other lines not containing $e$, so there are at least $26 - 20 = 6$ lines of $M$ disjoint from $e$. Let $B$ be a set formed by choosing of a pair of elements from each of these lines. In a similar manner to the previous lemma, we obtain six partitions of $B$ that contradict Lemma~\ref{computation}. We thus assume that $e \notin L$.
	
	Let $L = \{x_1, \dotsc, x_5\}$. Each $x \in L$ lies on at most four lines other than $L$ not containing $e$, so there exist $26 - 1 - 20 = 5$ lines $L_{0,1}, \dotsc, L_{0,5}$ of $M$ disjoint from $L \cup \{e\}$. If there are six such disjoint lines, then we again obtain a contradiction with Lemma~\ref{computation}; we therefore assume that every $x_i$ in $L$ lies on exactly four other lines of $M$ disjoint from $L$, so $\delta_M(x_i) = 6$ for each $i \in \{1,\dotsc,5\}$.
	
	 For each $j \in \{1, \dotsc, 5\}$ let $a_{2j-1},a_{2j}$ be distinct elements of $L_{0,j}$. Let $A = \{a_1, \dotsc, a_{10}\}$ and let $N = M|(L \cup A \cup \{e\})$. As in the proof of the previous lemma the lines $L_{0,j}$ partition $A$ into pairs, and so $|N| = 16$. Since $e$ lies on at most $5$ lines of $N$ each containing at most three other points, the elements of $E(N) - \{e\}$ partition into three-element sets $L_{1,e}, \dotsc, L_{5,e}$ such that $L_{j,e} \cup \{e\}$ is a four-element line of $N$ for each $j$.
	
	As before we consider the lines through each element of $L$, and for each $x_i \in L$ we obtain a partition $\cL_i = \{L_{i,1}, \dotsc, L_{i,5}\}$ of $A \cup \{e\}$ into five blocks corresponding to the lines of $N$ through $x_i$ other than $L$. Again we have $4 \ge |L_{i,1}| \ge |L_{i,2}| \ge \dotsc \ge |L_{i,5}| = 1$, (we have $|L_{i,5}| = 1$ here by minimality of $M$ and the fact that $\delta_M(x_i) = 6$) and $\sum_{j = 1}^5 |L_{i,j}| = 11$ for each $i$. Moreover, for each $i$ the point $x_i$ is on the four-element line $L_{i,e}$, so for some $j$ we have $|L_{i,j}| = 3$. Finally, there are $\binom{11}{2}-5 = 50$ pairs of elements in $A \cup \{e\}$ that do not span one of the lines $L_{0,i}$, so $\sum_{i = 1}^5\sum_{j = 1}^5\binom{|L_{i,j}|}{2} = 50$. 
	
	If $4 \ge n_1 \ge \dotsc \ge n_5 = 1$ are integers summing to $11$ such that some $n_i$ is $3$, then $\binom{n_1}{2} + \dotsc + \binom{n_5}{2} \le 10$ with equality only if $(n_1,n_2,\dotsc,n_5) = (4,3,2,1,1)$. Therefore  $(|L_{i,1}|,|L_{i,2}|, \dotsc, |L_{i,5}|) = (4,3,2,1,1)$ for each $i$; note that $L_{i,e} \cup \{x_i\} = L_{i,2}$. Therefore, in the fifteen-element matroid $N \del e$, each $x_i \in L$ lies on two five-element lines; two three-element lines and two two-element lines. For each integer $k$ Let $\cJ_k$ be the set of $k$-element lines of $N \del e$.
	
	Let $Y$ be the union of the lines in $\cJ_5$. By the above reasoning each $y \in Y$ lies on exactly two lines in $\cJ_5$, so it follows that $5|\cJ_5| = 2|Y|$ and so $|Y| \equiv 0 \pmod{5}$. Since three $5$-point lines account for at least $13$ points, it is clear that $|Y| >10$ and so we must have $|Y| = 15$ and $|Y| = E(N \del e)$. Therefore every element of $N \del e$ lies on exactly two lines in $\cJ_5$, $|\cJ_5| = \tfrac{2}{5}|Y| = 6$, and the elements of $N \del e$ are exactly the intersections of the $\binom{6}{2}$ pairs of lines in $\cJ_5$. There is now a natural mapping of $E(N \del e)$ to the edge set of the complete graph $K_6$ with vertex set $\cJ_5$, where the elements of each $J \in \cJ_5$ are the edges incident with the vertex $J$. The lines in $\cJ_3$ map to three-edge matchings. We know the lines $\cL_{i,e}- \{e\}$ are in $\cJ_3$ and partition $E(N \del e)$, and each $f \in E(N \del e)$ is contained in exactly two lines in $\cJ_3$, so $\cJ_3$ is the union of two disjoint partitions of $E(N \del e)$. This gives two disjoint $1$-factorisations of $K_6$, a contradiction by Lemma~\ref{k6}. 
	\end{proof}

\section{Higher Rank}

	Combining all lemmas in the last two sections gives the following: 
	
	\begin{theorem}\label{rank3}
		If $q \in \{2,3,4,5\}$ and $M$ is a rank-$3$ matroid in $\cU(q)$, then $W_2(M) \le q^2 + q + 1$ and $W_2^e(M) \le q^2$ for each nonloop $e$ of $M$. 
	\end{theorem}

	We now generalise this to arbitrary rank. For a matroid $M$ and a nonloop $e \in E(M)$, let $\cP_M(e)$ denote the set of planes of $M$ containing $e$. Note that $|\cP_M(e)| = W_2(M \con e)$. When we contract a nonloop $e$ in a matroid $M$, every line through $e$ becomes a point and every set of lines not containing $e$ that span a plane in $\cP_M(e)$ are identified into a single line. This gives the following lemma:

	\begin{lemma}\label{contract}
		If $M$ is a matroid and $e \in E(M)$ is a nonloop, then $W_2(M) = W_1(M \con e) + \sum_{P \in \cP_M(e)}W_2^e(M|P)$.
	\end{lemma}
	
	From here we can easily verify Conjecture~\ref{main} for all $q \le 5$. 
	
	\begin{theorem}
		If $q \in \{2,3,4,5\}$ and $M \in \cU(q)$ then $W_2(M) \le \qbinom{r(M)}{2}$. %and $r(M) = r$ then $W_2(M) \le \qbinom{r}{2}$.
	\end{theorem}
	\begin{proof}
		If $r \le 2$ then the result is obvious. Suppose inductively that $r \ge 3$ and that the result holds for smaller $r$, and let $e$ be a nonloop of $M$. By Theorem~\ref{kung} we have $W_1(M \con e) \le \frac{q^{r-1}-1}{q-1}$ and by Theorem~\ref{rank3} we have $W_2^e(M|P) \le q^2$ for each $P \in \cP_M(e)$. Therefore, by Lemma~\ref{contract} and the inductive hypothesis, 
		\begin{align*}
			W_2(M) &=  W_1(M \con e) + \sum_{P \in \cP_M(e)}W_2^e(M|P)\\
				 &\le \tfrac{q^{r-1}-1}{q-1} + q^2|\cP_M(e)| \\
				 &= \tfrac{q^{r-1}-1}{q-1} + q^2 W_2(M \con e) \\
				 &\le \tqbinom{r-1}{1} + q^2 \tqbinom{r-1}{2} \\
				 & = \tqbinom{r}{2},\\
		\end{align*} 
		as required. 
		
	\end{proof}

\section{References}

\newcounter{refs}

\begin{list}{[\arabic{refs}]}
{\usecounter{refs}\setlength{\leftmargin}{10mm}\setlength{\itemsep}{0mm}}

\item\label{gn}
J. Geelen, P. Nelson, 
The number of points in a matroid with no $n$-point line as a minor, 
J. Combin. Theory. Ser. B 100 (2010), 625--630.

\item\label{kungextremal}
J.P.S. Kung,
Extremal matroid theory, in: Graph Structure Theory (Seattle WA, 1991), 
Contemporary Mathematics 147 (1993), American Mathematical Society, Providence RI, ~21--61.

\item\label{nelson}
P. Nelson, 
The number of rank-$k$ flats in a matroid with no $U_{2,n}$-minor,
arXiv:1306.0531 [math.CO]

\item \label{oxley}
J. G. Oxley, 
Matroid Theory,
Oxford University Press, New York, 2011.

\end{list}
\end{document}